\documentclass{amsart}
\usepackage[utf8]{inputenc}
\usepackage{mathtools}
\usepackage{amsfonts}
\usepackage{xcolor}

\newtheorem{thm}{Theorem}

\newtheorem{lemma}[thm]{Lemma}
\newtheorem{rem}[thm]{Remark}
\newtheorem{example}[thm]{Example}

\usepackage{tikz, tkz-euclide, pgffor}

\DeclareMathOperator{\Vol}{Vol}

\newcommand{\mc}{\mathcal}
\newcommand{\del}{\partial}
\newcommand{\InPr}[2]{\langle #1, #2 \rangle}

\newcommand{\bnabla}{\overline{\nabla}}

\newcommand{\rest}[1]{\bigl|\bigr._{#1}}

\newcommand{\R}{\mathbb{R}}
\newcommand{\N}{\mathbb{N}}

\newcommand{\D}{\mathbb{D}}

\newcommand{\TM}{\tilde{M}}

\begin{document}
\title[Hypersurfaces with small Steklov eigenvalues]{Hypersurfaces of Euclidean space with prescribed boundary and small Steklov eigenvalues}
\date{}
\begin{abstract}
Given a smooth compact hypersurface $M$ with boundary $\Sigma=\partial M$, we prove the existence of a sequence $M_j$ of hypersurfaces with the same boundary as $M$, such that each Steklov eigenvalue $\sigma_k(M_j)$ tends to zero as $j$ tends to infinity. The hypersurfaces $M_j$ are obtained from $M$ by a local perturbation near a point of its boundary. Their volumes and diameters are arbitrarily close to those of $M$, while the principal curvatures of the boundary remain unchanged. 
\end{abstract}

\author{Bruno Colbois}
\address{Universit\'e de Neuch\^atel, Institut de Math\'ematiques, Rue
	Emile-Argand 11, CH-2000 Neuch\^atel, Switzerland}
\email{bruno.colbois@unine.ch}
\author{Alexandre Girouard}
\address{D\'e\-par\-te\-ment de math\'ematiques et de
	sta\-tistique, Univer\-sit\'e Laval, Pavillon Alexandre\-Vachon,
	1045, av. de la M\'edecine,
	Qu\'ebec Qc G1V 0A6, 
	Canada }
\email{alexandre.girouard@mat.ulaval.ca}

\author{Antoine Métras}
\address{D\'epartement de math\'ematiques et de statistique, Universit\'e de Montr\'eal, CP 6128 succ Centre-Ville, Montr\'eal, QC H3C 3J7, Canada}
\email{metrasa@dms.umontreal.ca}

\maketitle

\section{Introduction}
Let $M$ be a $n$-dimensional smooth compact Riemannian manifold with boundary $\Sigma=\partial M$.  The Steklov
eigenvalue problem on $M$ consists in finding all numbers $\sigma\in\R$ for which there exists a nonzero function $u \in C^\infty(M)$ which solves
\begin{equation*}
    \begin{cases}
        \Delta u = 0 & \text{in } M, \\
        \del_\nu u = \sigma u & \text{on } \Sigma.
    \end{cases}
\end{equation*}
Here $\Delta$ is the Laplacian induced from the Riemannian metric $g$ on $M$ and $\partial_\nu$
is the outward pointing normal derivative along the boundary $\Sigma$. The Steklov eigenvalues 
form an unbounded increasing sequence
$0 = \sigma_0 \leq \sigma_1 \leq \sigma_2 \leq \dots \to \infty$, each of which is repeated according to its multiplicity. 
See \cite{GP17} and \cite{legacy} for background on this problem.

One of our main interests in recent years has been to understand the particular role that the boundary  $\Sigma$ plays with respect to Steklov eigenvalues. Some papers studying this question are \cite{CGH18, PS16,CESG17,CG18,WX09,Kar15,GPPS,CGG17,shamma}. 
In particular, we have considered the effect of various geometric constraints on individual eigenvalues $\sigma_k$. One particularly interesting question is to prescribe a Riemannian metric $g_\Sigma$ on the boundary $\Sigma$ and to investigate lower and upper bounds for the eigenvalue $\sigma_k$ among all Riemmanian metrics $g$ which coincide with $g_\Sigma$ on the boundary. Given any Riemannian metric $g$ on $M$ such that $g=g_\Sigma$ on $\Sigma$, it is proved in \cite{CESG17} that one can make any eigenvalue $\sigma_k$ arbitrarily small by modifying the Riemannian metric $g$ in an arbitrarily small neighborhood $V\subset M$ of a point $p\in\partial M$. More precisely: for each $\epsilon>0$ and each $k\in\N$, there exists a Riemannian metric $\tilde{g}=\tilde{g}_{\epsilon,k}$ on $M$ which coincide with $g_\Sigma$ on $\Sigma$ and also with $g$ outside the neighborhood $V$, such that $\sigma_k(M,\tilde{g})<\epsilon$. For manifolds $M$ of dimension $n\geq 3$ one can also obtain arbitrarily large eigenvalues, but in general not using a perturbation that is localized near the boundary of $M$. See \cite{CESG17,CG18}. In \cite{CGG17} a more restrictive constraint was imposed by requiring the manifold $M$ to be a submanifold of $\R^m$ with prescribed boundary $\Sigma=\partial M\subset\R^m$. In this context an upper bound for $\sigma_k$ was given, in terms of $\Sigma$ and of the volume of $M$. The authors were unable at the time to give a lower bound and they raised the question to know if there exists one, or if instead arbitrarily small eigenvalues are possible. The goal of this paper is to answer this question.

\begin{thm}\label{thm:main}
Let $M \subset \R^{n+1}$ be a smooth $n$-dimensional compact hypersurface with nonempty boundary $\Sigma=\partial M$. For each $p \in \Sigma$, there exists a sequence of hypersurfaces $M_j\subset\R^{n+1}$, $j\in\N$, with boundary $\partial M_j=\Sigma$ such that,
\begin{gather}\label{limit:mainthm}
\lim_{j\to\infty}\sigma_k(M_j)=0\qquad\mbox{ for each }k\in\N.
\end{gather}
The hypersurface $M_j$ coincide with $M$ outside of a ball $B(p,\frac{1}{j})$ and the principal curvatures of $\Sigma\subset M_j$ are independent of $j$. Moreover, 
the volume and diameter of $M_j$ converge to those of $M$ as $j\to\infty$. 
\end{thm}

In order for \eqref{limit:mainthm} to hold for each $k\in\N$, it is necessary that the perturbed hypersurfaces $M_j$ differ from $M$ arbitrarily close to the boundary $\Sigma$ as $j\to\infty$. Indeed, let $b$ be the number of connected components of $\Sigma$ and note that any hypersurface $\tilde{M}$ which coincide with $M$ in a neighborhood $\Omega$ of $\Sigma$ satisfy
$\sigma_{b+1}\geq C>0$ where $C$ is given by a sloshing problem on $\Omega\cap M$. 
See \cite{CGG17} for details.

\begin{rem}
 Theorem 1 holds in arbitrary positive codimension. We decided to state it for hypersurfaces for the sake of notational simplicity.
\end{rem}

\begin{rem}
	Because each manifold $M_j$ coincide with $M$ in a neighborhood $\Omega_j$ of its prescribed boundary $\Sigma$, the Dirichlet-to-Neumann map $D_j:C^\infty(\Sigma)\rightarrow C^\infty(\Sigma)$ associated to $M_j$ all have the same full symbol \cite[section 1]{LeeUhlmann}. See also \cite{PoltSher}.
\end{rem}

\subsection{The strategy of the proof}
For eigenvalues of the Laplace operator, it is well known that one can obtain arbitrarily small eigenvalues by constructing thin Cheeger dumbbells in the interior of the manifold. See \cite{Chavel,CoursColboisMtl}. This strategy does not work for Steklov eigenvalues. For Steklov eigenvalues, it is possible to obtain arbitrarily small eigenvalues by creating thin channels, but this involves deformation of the boundary. See \cite{GP10}. In order to prove Theorem \ref{thm:main}, we have to use a more elaborate strategy.

Given a smooth function $\tilde{u}:\R^{n+1}\rightarrow\R$, consider the restriction $u=\tilde{u}\rest{M}$.
It is well known that if $\int_{\Sigma}u\,dA=0$,
\begin{gather}\label{ineq:variationalintro}
\sigma_1(M)\int_{\Sigma}u^2\,dA\leq\int_M|\nabla u|^2\,dV.
\end{gather}
See \cite{GP17} and Section \ref{section:prelim} below. Here $\nabla u$ is the tangential gradient of $u$. It is the projection of the ambient gradient $\bnabla\tilde{u}$ on the tangent spaces of $M\subset\R^{n+1}$. The basic idea of our proof is to fix a function $\tilde{u}\in C^\infty(\R^{n+1})$ and consider the vector field $\bnabla\tilde{u}$ in the ambient space $\R^{n+1}$. The hypersurface $M$ is then deformed by creating ``wrinkles"  that tend to make the various tangent spaces $T_qM$, for $q\in\mbox{int }M$, perpendicular to $\bnabla\tilde{u}(p)$. This is achieved by ``folding the surface like an accordion" in the direction perpendicular to $\bnabla\tilde{u}$. In the limit the right-hand-side of inequality \eqref{ineq:variationalintro} tends to zero. 
Let us illustrate this strategy with a simple example.
\begin{example}
Given a smooth function $f:\overline{\D}\rightarrow\R$ vanishing on the circle $S^1=\partial\D$, consider the surface
$$S_f:=\mbox{Graph of }f=\{(x,y,f(x,y))\,:\,(x,y)\in\overline{\D}\}.$$
The boundary of $S_f$ is the same for each $f$.
We will use the function defined by $\tilde{u}(x,y,z)=x$ and its restriction $u=\tilde{u}\rest{S_f}$ as a trial function in inequality \eqref{ineq:variationalintro}. Because $\nabla\tilde{u}=(1,0,0)$, it follows from Lemma \ref{lemma:dirichletgraph} that the Dirichlet energy of $u:=\rest{S_f}\tilde{u}:S_f\rightarrow\R$ is given by
$$\int_{S_f}|\nabla u|^2=\int_{\D} \frac{1 + f_y^2}{\sqrt{1 + f_x^2+f_y^2}}
\,dxdy.$$
For $n\in\N$, define $f=f_n:\overline{\D}\rightarrow\R$ by
$$f(x,y)=\sin(nx)(\overbrace{1-x^2-y^2}^{\phi(x,y)}).$$
It follows from
\begin{gather*}
f_x^2=n^2\bigl(\cos(nx)\phi+\frac{1}{n}\sin(nx)\phi_x\bigr)^2\quad\mbox{ and }\quad
f_y^2=\sin^2(nx)\phi_y^2
\end{gather*}
that
$$\lim_{n\to\infty}\int_{S_{f_n}}|\nabla u|^2=0.$$
Together with \eqref{ineq:variationalintro}, this shows that
$\lim_{n\to\infty}\sigma_1(S_{f_n})=0.$
\end{example}

The proof of Theorem \ref{thm:main} is based on the above idea, but it is technically more involved because we want to localize this argument to a small neighbourhood of a point $p$ of the boundary. This is a significative gain compared to the above example because it allows the construction of an arbitrary finite number of disjointly supported test functions with small Dirichlet energy, leading to the collapse of each eigenvalue $\sigma_k$ rather than just $\sigma_1$. For the sake of readability and simplicity, the deformation that we use in the proof of Theorem \ref{thm:main} are continuous but only piecewise smooth. This is not problematic because only first derivatives of these deformations appear.

\subsection*{Plan of the paper}
In Section \ref{section:prelim} we review the min-max characterization of Steklov eigenvalues and we prove a lemma regarding the control of the Dirichlet energy under quasi-isometries. We then proceed to construct the perturbed hypersurfaces in Section \ref{section:perturbation}. We use a quasi-isometric chart to an hypersurface with a flat boundary. The perturbed submanifold is then constructed by considering the graph
of a locally supported oscillating function. Finally, in Section \ref{section:testfunction} an appropriate test function is used to conclude the proof of Theorem \ref{thm:main}.

\section{Notations and preliminary considerations}\label{section:prelim}
Let $M$ be a smooth compact manifold with boundary $\Sigma$. The volume form on $M$ is written $dV$, while 
the volume form on $\Sigma$ is $dA$. We use the usual Sobolev space $H^1(M)\subset L^2(M,dV)$.
The Steklov eigenvalues $\sigma_k$ admits a variational characterization in terms of the \emph{Steklov-Rayleigh quotient} of a function $0\neq u \in H^1(M)$,
\begin{equation*}
    \mc{R}(u) = \frac{\int_{M} |\nabla u|^2\,dV}{\int_{\Sigma} u^2\,dA}.
\end{equation*}
The numerator $D(u)=\int_{M} |\nabla u|^2\,dV$ is the \emph{Dirichet energy} of $u\in H^1(M)$.
It is well known that
\begin{equation}\label{eq:minmaxsigmak}
    \sigma_k = \min_{\substack{S \subset H^1(M) \\ \dim S = k+1}}
    \max_{u \in S \setminus \{0\}} \mc{R}(u),
\end{equation}
where the minimum is taken over all $(k+1)$-dimensional linear subspaces $S\subset H^1(M)$.
\subsection{Quasi-isometries and Dirichlet energy}
Let $M$ and $\TM$ be two $n$-dimensional Riemannian manifolds with boundary. A diffeomorphism $\phi:M\rightarrow\TM$ is a \emph{quasi-isometry} with constant $C\geq 1$ if for each $p\in M$ and each $0\neq v\in T_pM$, 
$$\frac{1}{C}\leq\frac{\|D_p\phi(v)\|^2}{\|v\|^2}\leq C.$$
Quasi-isometries provide a control of the Dirichlet energy of a function.
\begin{lemma}\label{lemma:QIcontrolEnergy}
Let $\phi:M\rightarrow\TM$ be a quasi-isometry with constant $C\geq 1$. Let $f\in H^1(\TM)$ then
$$\frac{1}{C^{\frac{n}{2}+1}}\leq\frac{\|\nabla (f\circ\phi)\|_{L^2(M)}^2}{\|\nabla f\|_{L^2(\TM)}^2}\leq C^{\frac{n}{2}+1}.$$
\end{lemma}
\begin{proof}
Let $\tilde{g}$ be the Riemannian metric of $\TM$ and let $g$ be that of $M$. Let $\hat{g}=\phi^\star(\tilde{g})$. Because $\phi$ is a quasi-isometry with constant $C$, the following holds for each $0\neq v\in TM$,
$$\frac{1}{C}\leq \frac{g(v,v)}{\hat{g}(v,v)}\leq C.$$
It follows that
$$\frac{1}{C}g(\nabla_{g}f,\nabla_{g}f)\leq\hat{g}(\nabla_{\hat{g}}f,\nabla_{\hat{g}}f)\leq Cg(\nabla_{g}f,\nabla_{g}f).$$
The corresponding volume forms satisfy
$$C^{-n/2}dV_g\leq dV_{\hat{g}}\leq C^{n/2}dV_g.$$
This leads to
\begin{align*}
\|\nabla (f\circ\phi)\|_{L^2}^2&=\int_M g\bigl(\nabla_g (f\circ\phi),\nabla_g(f\circ\phi)\bigr)\,dV_g\\
&\leq
C^{n/2+1}\int_M {\hat{g}}\bigl(\nabla_{\hat{g}} (f\circ\phi),\nabla_{\hat{g}} (f\circ\phi)\bigr)\,dV_{\hat{g}}\\
&=
C^{n/2+1}\int_{\TM} {\tilde{g}}(\nabla_{\tilde{g}} f,\nabla_{\tilde{g}} f)\,dV_{\tilde{g}}.
\end{align*}
The lower bound is identical.
\end{proof}
\subsection{Quasi-isometric charts}
Recall that a subset $M\subset\R^{n+1}$ is an hypersurface with boundary if for each $p\in M$, there exist open sets $W,W'\subset\R^{n+1}$ with $p\in W$ and a diffeomorphism $\psi:W\rightarrow W'$ such that
$\psi(M\cap W)$ is an open set in the half-space
$$H=\{x\in\R^{n+1}\,:\,x_{n+1}=0, x_1\geq 0\}.$$
The point $p\in M$ is on the boundary $\Sigma$ of $M$ if and only if $\psi$ sends it to the boundary of the half-space $H$: 
$$\psi(p)\in\partial H:=\{x\in H\,:\,x_1=0\}.$$
This definition is coherent: it does not depend on the choice of the diffeomorphism $\psi$.
See \cite{Hi76} for further details.
By further restricting $\psi$ and scaling if necessary, we can assume that it is a quasi-isometry and that its image $W'$ is a cylinder.
\begin{lemma}\label{lemma:flatchart}
	For each $p\in\Sigma$, there exists a quasi-isometry 
	$$\psi:W\longrightarrow W'=B_{\R^n}(0,1)\times (-1,1)$$
	with $\psi(p)=0$ and such that the image of $M\cap W$ is
	$$U:=\psi(M\cap W)=\{x\in H\,:\,|x|<1\}.$$
\end{lemma}
\begin{rem}
	We identify $U\subset\R^{n+1}$ with a subset of $\R^n$ so that we can write $x=(x_1,\cdots,x_n)\in U$ instead of
	$x=(x_1,\cdots,x_n,0)\in U$.
\end{rem}
\begin{figure}[h]
	\centering
	\includegraphics[width=6cm]{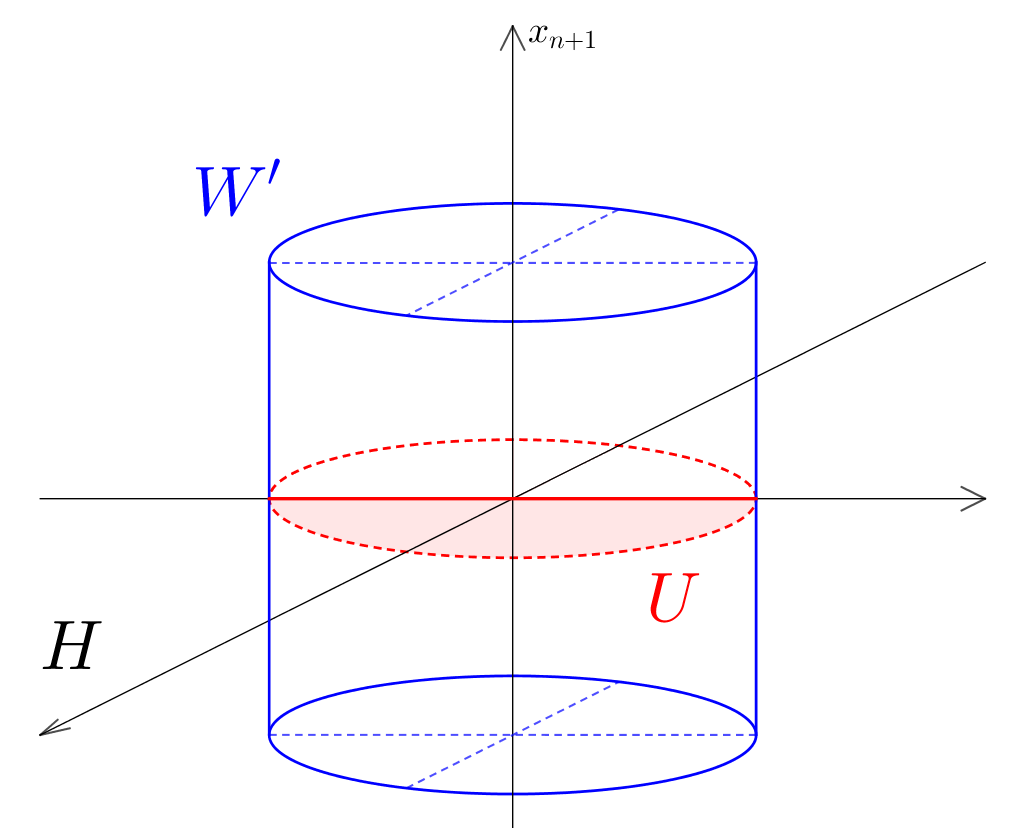}
	\caption{The domain $U$}
\end{figure}
In particular, the restriction of $\psi$ to $\Sigma\cap W$ is also a quasi-isometry.

\subsection{Dirichlet energy on the graph of a function}
Let $U\subset\R^n$ be a bounded open set and let $f:U\rightarrow\R$ be a bounded smooth function. Consider the graph
$$S_f=\{(x,f(x))\,:\,x\in U\}\subset\R^{n+1}.$$
Given a function $u:U\rightarrow\R$, define $\tilde{u}:U\times\R\rightarrow\R$ by $\tilde{u}(x,x_{n+1})=u(x)$ and define $u_f:S_f\rightarrow\R$ is defined by
$$u_f(x,f(x)) = u(x)=\tilde{u}\rest{S_f}.$$
\begin{lemma}\label{lemma:dirichletgraph}
	The Dirichlet energy of $u_f:S_f\rightarrow\R$ is
	\begin{equation*}
	\int_{S_f} |\nabla u_f|^2\,dV
	= \int_U \frac{|\nabla u|^2 + |\nabla u|^2 |\nabla f|^2
		- \InPr{\nabla u}{\nabla f}^2}{\sqrt{1 + |\nabla f|^2}}
	\,dx,
	\end{equation*}
	where on the left-hand-side $\nabla$, $dV$ and the norm are taken on $S_f$ and on the right-hand-side $dx=dx^1\cdots dx^n$ is the Lebesgue measure on $U$, while $\nabla$ is the usual gradient on $\R^n$.
\end{lemma}
	The Cauchy-Schwarz inequality gives $|\nabla u|^2 |\nabla f|^2
- \InPr{\nabla u}{\nabla f}^2 \geq 0$ with equality if and only
if $\nabla u = c \nabla f$ for some constant $c$.
\begin{proof}[Proof of Lemma \ref{lemma:dirichletgraph}]
	To simplify notations, we will write $S=S_f$. 
	For any point $p\in S$, the gradient $\nabla u_f\in T_pS$ is the projection of $\bnabla\tilde{u}$ on $T_pS$.
	That is,
	\begin{equation*}
	\nabla u_f = \bnabla \tilde{u} - \InPr{\bnabla \tilde{u}}{N}N,
	\end{equation*}
	where $N$ is a unit normal vector to $T_pS$.
	It follows from
	\begin{equation*}
	\nabla\tilde{u} = \left(\frac{\del u}{\del x_1}, \dots, \frac{\del u}
	{\del x_n}, 0\right)
	\end{equation*}
	and
	\begin{equation}
	N =\frac{1}{\sqrt{1 + |\nabla f|^2}}\left(e_{n+1}-\sum_{i=1}^n\frac{\partial f}{\partial x_i}e_i\right).
	\end{equation}
	that
	\begin{gather}\label{identity:nablauf}
	|\nabla u_f|^2 
	= \frac{|\nabla u|^2 + |\nabla u|^2|\nabla f|^2
		- \InPr{\nabla u}{\nabla f}^2}{1 + |\nabla f|^2}.
	\end{gather}
	
	The volume element on $S_f$ is given by

	\begin{equation}
	dV 
	= \sqrt{1 + |\nabla f|^2}
	\end{equation}
	Together with identity \eqref{identity:nablauf} this completes the proof.
\end{proof}

\section{Perturbation of the submanifold $M$}\label{section:perturbation}
Given $p\in\Sigma$, let $\psi$ be the quasi-isometric chart provided by Lemma \ref{lemma:flatchart}. In order to prove Theorem \ref{thm:main}, we will deform the submanifold $M$ in the neighborhood $W$ of the point $p$ by deforming the neighborhood $U\subset W'$ inside $W'$ and pulling back to $W$ using the quasi-isometry $\psi$.
Consider a smooth function $f:U\rightarrow\R$ which is supported in the interior of $U$ and which satisfy $|f(x)|<1$ for each $x\in U$. This last condition implies that the graph of $f$,
$$S_f = \{(x, f(x))\,:\,x\in U\}$$
is contained in the cylinder $W'$. 
Hence it can be used to define a deformation of $M$ as follows:
\begin{gather}\label{def:TM}
\TM_f:=(M\setminus W)\cup\psi^{-1}(S_f).
\end{gather}
Because $f$ is smooth, supported in $U$ and $S_f\subset W'$, the subset $\TM_f\subset\R^{n+1}$ also is a submanifold with boundary $\partial\TM_f=\Sigma=\partial M$.

\subsection{Deformation function}

We now construct specific functions $f$ and $u$ such that the Dirichlet energy of $u_f$ is small, basing our method on the idea that if $\nabla u$
and $\nabla f$ are parallel the numerator is independent of $\nabla f$,
while the denominator behaves as $|\nabla f|$. Thus we want $f$ and $u$
to have parallel gradients with $|\nabla f|$ big to get a small Dirichlet
energy for $u_f$.

Consider numbers
$\epsilon, \delta_1, \delta_2, \rho > 0$ that are sufficiently small and define
$\delta:=\delta_1 + \delta_2$. These constants will be adjusted later in equation \eqref{dfn:constants}.
Let $q = (\delta, 0, \dots, 0)\in H$. Consider the following subsets of $U$:
\begin{align*}
  A &:= \{(x_1, \dots, x_n) \,|\, x_1 \geq \delta, \|x - q\| \leq \epsilon\} \\
  B &:= \{(x_1, \dots, x_n) \,|\, \delta_1 \leq x_1 \leq \delta,
      \|\pi x\| \leq \epsilon\} \\
  C &:= \{(x_1, \dots, x_n) \,|\, 0 \leq x_1 \leq \delta_1,
      \|\pi x\| \leq \epsilon\} \\
  D &:= \{x \in U \setminus (A \cup B \cup C) \,|\, d(x,A \cup B \cup C)
      \leq \rho\}
\end{align*}
where $\pi(x_1, x_2, \dots, x_n) = (x_2, \dots, x_n)$ is the projection
on the boundary. Let $\Omega = A \cup B \cup C$. See Figure \ref{fig:domain}, where the $x_1$-axis is vertical.
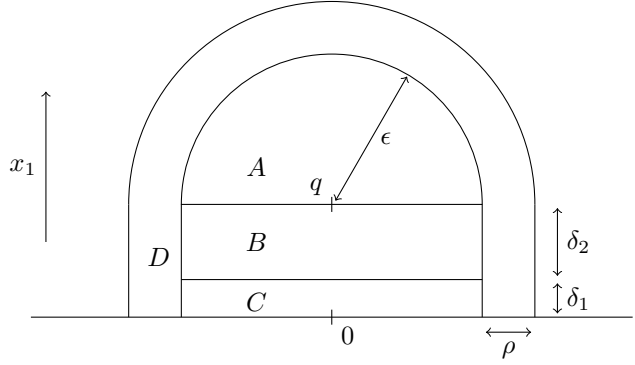
\begin{figure}[h]
  \centering
  \begin{tikzpicture}
    \draw [->] (-3.8, 1) -- node [left]{$x_1$} (-3.8,3);

    \coordinate [label=below right:$0$] (o) at (0,0);
    \draw (0,0.1) -- (0, -0.1);
    \draw (-4,0) -- (4,0);

    \coordinate [label=above left:$q$] (p) at (0, 1.5);
    \draw (0,1.6) -- (0,1.4);
    \draw (2,1.5) arc [start angle=0, end angle=180, radius=2];
    \draw [<->] (0.05, 1.55) -- node [right]{$\epsilon$} (1, 3.2);
    \node at (-1,2) [] {$A$};

    \draw (-2,0) -- (-2,1.5) -- (2,1.5) -- (2,0);
    \draw (-2,0.5) -- (2,0.5);
    \node at (-1,1) [] {$B$};
    \node at (-1, 0.2) [] {$C$};
    \draw [<->] (3,0.05) -- node [right]{$\delta_1$}(3,0.45);
    \draw [<->] (3,0.55) -- node [right]{$\delta_2$}(3,1.45);

    \draw (2.7,0) -- (2.7,1.5);
    \draw (2.7,1.5) arc [start angle=0, end angle=180, radius=2.7];
    \draw (-2.7,1.5) -- (-2.7,0);
    \node at (-2.3,0.8) [] {$D$};
    \draw [<->] (2.05,-0.2) -- node [below]{$\rho$} (2.65,-0.2);

  \end{tikzpicture}
  \caption{Perturbation region}
  \label{fig:domain}
\end{figure}

Define cutoff functions
\begin{align*}
  \eta &: [0, \infty) \to \mathbb{R} & \gamma &: [0, \delta]
                                              \to \mathbb{R} \\
  \eta(x) &= \max\{0, 1- x\} &
  \gamma(x) &= \begin{cases}
    0 & \text{if } x \leq \delta_1 \\
    x - \delta_1 & \text{if } \delta_1 \leq x \leq \delta
  \end{cases}.
\end{align*}
Finally, define $F : [0, \infty) \to \mathbb{R}$ to be periodic
of period 4, given on the interval $[0,4]$ by
\begin{align*}
  F(x) = \begin{cases}
    1 - x & \text{if } x \in [0,2],\\
    x - 3 & \text{if } x \in [2,4].
  \end{cases}
\end{align*}
Given $\omega>0$, define $f : U \rightarrow \mathbb{R}$ by
\begin{equation}
  f(x) = \begin{cases}
    \eta\left(\frac{d(x,\Omega)}{\rho}\right)
    \gamma(x_1) F(\omega \|\pi x\|) & \text{if } x_1 \leq \delta, \\
    \eta\left(\frac{d(x, \Omega)}{\rho}\right)
    \delta_2 F(\omega \|x - q \|) & \text{if } x_1 \geq \delta.
  \end{cases}
\end{equation}
Note that functions $\eta$ and $\gamma$ are used to localize the deformation function $f$. The parameter $\omega$ will be sent to $\infty$ later in the proof. 
It is important to remark that $|F'(x)| = 1$ at points where
$F$ is differentiable and $|F(x)| \leq 1$ for all $x$.
\begin{rem}\label{rem:smoothing}
The deformation function $f$ are continuous and piecewise smooth. Because only first order derivatives of these functions appear in the estimates below, one could replace them by smooth approximations without affecting the results.
\end{rem}

\section{Test function}\label{section:testfunction}
The test function $u$ is supported on $\Omega\subset U$
and is defined by
\begin{equation}
  u(x) = \begin{cases}
    1 - \frac{\|\pi x\|}{\epsilon} \quad \text{if } x \in B \cup C \\
    1 - \frac{\|x - q\|}{\epsilon} \quad \text{if } x \in A 
  \end{cases}
\end{equation}

On $A$, the Dirichlet energy of $u_f$ can be made small by taking
$\omega$ big. Indeed for almost all $x \in A$,
\begin{align*}
  \nabla f &= \pm \delta_2 \omega \frac{x - q}{\|x - q\|} \\
  \nabla u &= - \frac{1}{\epsilon}\frac{x - q}{\|x - q\|}
\end{align*}
and using the fact that $\nabla f$ and $\nabla u$ are parallel,
the Dirichlet energy is
\begin{align*}
  \int_{S_f \cap A \times \mathbb{R}} |\nabla u_f|^2 dV
  &= \int_A \frac{|\nabla u|^2}{\sqrt{1 + |\nabla f|^2}}\,dx
  = \frac{1}{\epsilon^2} \frac{1}{\sqrt{1 + \delta_2^2 \omega^2}} \Vol A
  = \frac{c_1 \epsilon^{n-2}}{\sqrt{1 + \delta_2^2 \omega^2}}
\end{align*}
where $c_1$ is some constant.

On $B$ and $C$, $\nabla f$ and $\nabla u$ are not parallel but
it is possible to make the Dirichlet energy small by making
the volume of $B$ and $C$ small:
for $B$, we have for almost all $x \in B$:
\begin{align*}
  \nabla f(x) &= F(\omega \|\pi x\|) e_1 \pm \gamma(x_1) \omega
                \frac{\pi x}{\|\pi x\|} \\
  \nabla u(x) &= -\frac{1}{\epsilon} \frac{\pi x}{\|\pi x\|}
\end{align*}
and since $e_1$ and $\pi x$ are orthogonal,
\begin{equation*}
  |\nabla f(x)|^2 = F(\omega \|\pi x\|)^2 + \gamma(x_1)^2 \omega^2.
\end{equation*}
Then the Dirichlet energy on $B$ is
\begin{align*}
  \int_{S_f \cap B \times \mathbb{R}} |\nabla u_f|^2 dV
  &= \int_B \frac{\frac{1}{\epsilon^2} + \frac{1}{\epsilon^2}
    F(\omega \|\pi x\|)^2}
    {\sqrt{1 + F(\omega \|\pi x\|)^2 + \gamma(x_1)^2 \omega^2}}
    dx\\
  &\leq \frac{1}{\epsilon^2}
    \int_B \frac{2}{\sqrt{1 + \gamma(x_1)^2 \omega^2}} dx \\
  &= c_2 \epsilon^{n-3}
    \int_0^{\delta_2} \frac{1}{\sqrt{1 + x_1^2 \omega^2}} dx_1 \\
  &= c_2 \epsilon^{n-3}
    \frac{\ln(\delta_2 \omega + \sqrt{1 + \delta_2^2})}{\omega} \\
\end{align*}
And on $C$, since $f=0$, the Dirichlet energy is simply
\begin{equation*}
  \int_{S_f \cap C \times \mathbb{R}} |\nabla u_f|^2 dV = \frac{1}{\epsilon^2} \Vol C
  = c_3 \delta_1 \epsilon^{n-3}
\end{equation*}
In total, the Rayleigh quotient of $u_f$ is bounded as follows
\begin{align*}
  \mc{R}(u_f)
  &\leq \frac{1}{\epsilon^{n-1}}
    \left(c_1 \frac{\epsilon^{n-2}}{\sqrt{1 + \delta_2^2 \omega^2}}
    + c_2 \frac{\epsilon^{n-3} \ln(\delta_2 \omega +
    \sqrt{1 + \delta_2^2 \omega^2})}{\omega}
    + c_3 \delta_1 \epsilon^{n-3} \right)\\
  &= c_1 \frac{\epsilon^{-1}}{\sqrt{1 + \delta_2^2 \omega^2}}
    + c_2 \frac{\epsilon^{-2} \ln(\delta_2 \omega + \sqrt{1 +
    \delta_2^2 \omega^2})}{\omega}
    + c_3 \delta_1 \epsilon^{-2}.
\end{align*}
We are now ready to define the constants more precisely. By using the following:
\begin{gather}\label{dfn:constants}
\delta_1 = \epsilon^3,\, \delta_2 = \epsilon^{3/2},\, \omega = \epsilon^{-3}
\end{gather}
we obtain
\begin{equation*}
  \mc{R}(u_f) = \mc{O}(\epsilon^{1/2}) \quad \text{as } \epsilon \to 0
\end{equation*}
We have proved that a local perturbation of $U$ allows the construction
of a local trial function with arbitrarily small Steklov-Rayleigh quotient. The proof of our main result is now an easy consequence.
\begin{proof}[Proof of Theorem \ref{thm:main}]
  Let $k,j \in \mathbb{N}$ wit $j$ sufficiently large. Let $p_1, \dots, p_{k+1} \in B(p, \frac{1}{j}) \cap \Sigma$
  and let $\psi$ be the quasi-isometric chart from Lemma \ref{lemma:flatchart}. 
  For each $p_i$, we follow the above construction to obtain a deformation function $f_i$ and a test function $u_i$
  that is supported in a small enough neighbourhood of $p_i$ so that all the $u_i$'s have disjoint supports and
  are supported in $\psi(B(p,\frac{1}{j}) \cap M)$. The parameter $\epsilon$ in the previous construction small enough
  to guarantee that the Rayleigh quotient of each $u_i$ is smaller than $\frac{1}{j}$. 
  Consider the deformation function $f = f_1 + \dots + f_{k+1}$ supported in $B(p, \frac{1}{j})$ and the perturbed manifold $M_j = \TM_f$.
  Taking the pullback by $\psi$, we obtain $k+1$ test functions $\psi^*(u_i)$ with disjoint supports and
  from Lemma \ref{lemma:QIcontrolEnergy}, their
  Rayleigh quotient is less than $\frac{c}{j}$ where $c$ is a constant depending on $\psi$.
  By the variational characterization \eqref{eq:minmaxsigmak} of the eigenvalue $\sigma_k$, we conclude that $\sigma_k(M_j) \leq \frac{c}{j}$.

  It remains to prove that the perturbed manifolds $M_j$ satisfy the geometric conditions from the theorem.
  It is enough to show it for a perturbation around a single point $f = f_1$ and that the conditions
  are true when $\epsilon \to 0$.
 Let $x \in \Omega$. There exists $x' \in \Omega$ such that
$f(x') = 0$ and the distance in $M_f$ between $x$ and $x'$ is
$\mc{O}(\epsilon^{3/2})$.
There is a path from $x'$
to some point $y$ on $\del M$ such that the length of the path is less
than $\delta + \epsilon\pi/2 $, it suffices to
take the shortest path in $\{x \in M \,|\, f(x) = 0\}$ from
$x'$ to $\del M$ (see figure \ref{fig:path}).
This total length of the path from $x$ to $y$ goes to 0 when
$\epsilon$ goes to $0$ and since $\del \TM_f = \del M$ this implies
that the diameter of $\TM_f$ converges to $\mbox{Diam}(M)$ when $\epsilon$ goes to 0.

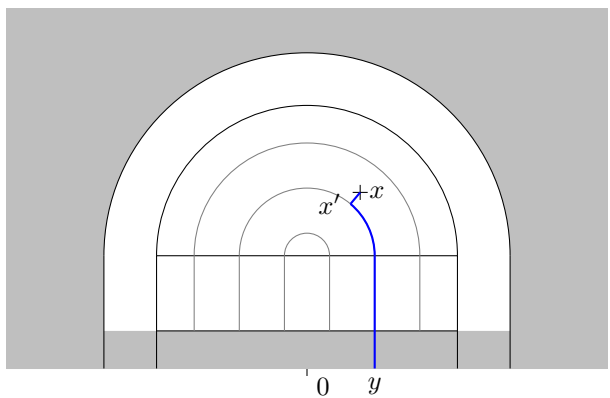
\begin{figure}
  \centering
  \begin{tikzpicture}
    \coordinate [label=below right:$0$] (o) at (0,0);
    \draw (0,0.1) -- (0, -0.1);
    \draw (-4,0) -- (4,0);

    \fill[lightgray] (-4,0) -- (-4,4.8) -- (4,4.8) -- (4,0) -- cycle;
    \fill[white] (2.7,0.5) -- (2.7,1.5)
    -- (2.7,1.5) arc [start angle=0, end angle=180, radius=2.7]
    -- (-2.7, 1.5) -- (-2.7,0.5) -- cycle;
    
    \draw (2,1.5) arc [start angle=0, end angle=180, radius=2];

    \draw (-2, 0) -- (-2,1.5) -- (2,1.5) -- (2,0);
    \draw (-2,0.5) -- (2,0.5);

    \draw (2.7,0) -- (2.7, 1.5);
    \draw (2.7,1.5) arc [start angle=0, end angle=180, radius=2.7];
    \draw (-2.7,1.5) -- (-2.7,0);

    \foreach \r in {0.3,0.9,...,1.8} {
      \draw[gray] (\r,1.5) arc [start angle=0, end angle=180, radius=\r];
      \draw[gray] (\r,0.5) -- (\r,1.5);
      \draw[gray] (-\r,0.5) -- (-\r,1.5);
    }

    \draw (50:1.1) ++(-0.1,1.5) -- node [right]{$x$} +(0.2,0);
    \draw (50:1.1) ++(0,1.4) -- +(0,0.2);
    \draw (50:0.9) ++(0,1.5) node[left]{$x'$};
    \draw[blue,thick] (50:0.9) ++(0,1.5) -- + (50:0.2);
    \draw[blue,thick] (50:0.9) ++(0,1.5) arc [start angle=50, end angle=0,
    radius=0.9];
    \draw (0.9,0) node[below]{$y$};
    \draw[blue,thick] (0.9,1.5) -- (0.9,0);

  \end{tikzpicture}
  \caption{The set $\{x \,|\, f(x) = 0\}$ is shown in grey. The path
    in blue from an arbitrary $x \in \Omega$ to a point on $\del \TM_f$
    has length going to 0 as $\epsilon \to 0$.}
  
  \label{fig:path}
\end{figure}

For the volume of $\TM_f$, taking $\rho = \epsilon$, the volume difference between
$\TM_f$ and $M$ goes to 0 as $\epsilon \to 0$. Indeed, using the fact that the chart $\psi$ is
a quasi-isometry, it is enough to show that the difference in volume between $S_f$ and
$\Omega \cup D$ goes to 0:
\begin{align*}
  |\Vol(S_f) - \Vol(\Omega \cup D)| &= \int_{\Omega \cup D}
  \sqrt{1 + |\nabla f|^2} dx_1 \dots dx_n  - \Vol(\Omega \cup D)\\
  &\leq (\sqrt{1 + \delta_2^2 \omega^2} - 1) \Vol(\Omega \cup D) \\
  &= \mc{O}(\epsilon^{n-1})
\end{align*}
which goes to 0 for $n \geq 2$. Finally, it is clear that the curvatures of
$\del \TM_f$ do not change as $M$ is kept fixed on some neighborhood
of the boundary (the set $C$).
\end{proof}

\bibliographystyle{plain}
\bibliography{biblio}

\end{document}